\documentclass[smallextended]{svjour3}%
\usepackage{amsmath}
\usepackage{amsfonts}
\usepackage{amssymb}
\usepackage{graphicx}
\usepackage{color}%
\providecommand{\U}[1]{\protect\rule{.1in}{.1in}}
\newtheorem{algorithm}{Algorithm}
\newtheorem{condition}{Condition}
\journalname{Numerical Algorithms}

\begin{document}

\title{Family Constraining of Iterative Algorithms}
\author{Yair Censor         \and
        Ioana Pantelimon 		\and
        Constantin Popa
}

\institute{Y. Censor \at
              Department of Mathematics, University of Haifa, Mt.\ Carmel, Haifa 3190501, Israel \\
              \email{yair@math.haifa.ac.il}          
           \and
           I. Pantelimon \at
              Faculty of Mathematics and Computer Science, Ovidius University, Blvd. Mamaia 124, Constanta 900527, Romania\\
              \email{ipomparau@univ-ovidius.ro} 
           \and
           C. Popa \at
              Faculty of Mathematics and Computer Science, Ovidius University, Blvd. Mamaia 124, Constanta 900527, Romania\\
              ``Gheorghe Mihoc - Caius Iacob'' Institute of Statistical Mathematics and Applied Mathematics, Calea 13 Septembrie, Nr. 13, Bucharest 050711, Romania\\
              \email{cpopa@univ-ovidius.ro} 
}

\date{Received: date / Accepted: date\\Final revision: July 5, 2013.}

\maketitle

\begin{abstract}
In constraining iterative processes, the algorithmic operator of the iterative
process is pre-multiplied by a constraining operator at each iterative step.
This enables the constrained algorithm, besides solving the original problem,
also to find a solution that incorporates some prior knowledge about the
solution. This approach has been useful in image restoration and other image
processing situations when a single constraining operator was used. In the
field of image reconstruction from projections a priori information about the
original image, such as smoothness or that it belongs to a certain closed
convex set, may be used to improve the reconstruction quality. We study here
constraining of iterative processes by a family of operators rather than by a
single operator.



\end{abstract}

\section{Introduction\label{sec0}}

This paper is about \textit{constraining of iterative processes} which has the
following meaning. When dealing with a real-world problem it is sometimes the
case that we have some prior knowledge about features of the solution that is
being sought after. If possible, such prior knowledge may be formulated as an
additional constraint and added to the original problem formulation. But
sometimes, when we have already at our disposal a \textquotedblleft
good\textquotedblright\ algorithm for solving the original problem without
such an additional constraint, it is beneficial to modify the algorithm,
rather than the problem, so that it will, in some way, \textquotedblleft take
care\textquotedblright\ of the additional constraint (or constraints) without
loosing its ability to generate (finitely or asymptotically) a solution to the
original problem. This is called constraining of the original iterative
algorithm. Given an (algorithmic) operator $Q:R^{n}\rightarrow{R}^{n}$ between
Euclidean spaces, the original iterative process may have the form%
\begin{equation}
x^{k+1}=Q(x^{k}),\text{ for all }k\geq0,\label{eq:stat}%
\end{equation}
under various assumptions on $Q.$ Constraining such an algorithm with a family
of operators means that we desire to use instead of (\ref{eq:stat}) the
iterative process%
\begin{equation}
x^{k+1}=S_{k}Q(x^{k}),\text{ for all }k\geq0,
\end{equation}
where $\{S_{k}\}_{k=0}^{\infty}$ is a family of operators $S_{k}%
:R^{n}\rightarrow{R}^{n},$ henceforth called the \textit{constraining
operators}.

Our purpose is to study the possibility to constrain an algorithm with a
family of operators and to analyze the asymptotic behavior of such
family-constrained algorithms. We extend earlier results on this topic that
were limited to a single constraining operator, i.e., $S_{k}=S$ for all
$k\geq0,$ see, e.g., \cite{b04,c09,e07,kbms85,kbsm91,kn85,kl90,smr81,yw82},
introducing a family of strictly nonexpansive operators $\{S_{k}%
\}_{k=0}^{\infty}$ and proving the convergence of the family-constrained
algorithms in a more general setting.

The paper is organized as follows. In Section \ref{sec0bis}, for a family of
strictly nonexpansive operators with nonempty common fixed points set and a
supplementary image reconstruction condition we adapt some results from
\cite{DPI90} for our purpose and prove our main convergence result. We present
the \textit{family-constrained algorithm} (FCA) in Section \ref{sec1} and we
prove that the series expansion methods and the smoothing matrices used in
\cite{DPI90} obey all our hypotheses. In Section \ref{sec2} we show that the
general iterative method introduced in \cite{npps}, which includes the
algorithms of Kaczmarz, Cimmino and Diagonal Weighting (see, e.g.,
\cite{tanabe}, \cite{npps} and \cite{cpopaDW}, respectively) as special cases,
is itself an algorithmic operator of the form required here and we give an
example of a family of nonlinear constraining operators which satisfy our assumptions.

\subsection{Relation with previous work}\label{sec:prev}

Some earlier works on this topic were limited to a single constraining
operator, i.e., $S_{k}=S$ for all $k\geq0,$ see, e.g.,
\cite{b04,c09,e07,kbms85,kbsm91,kn85,kl90,smr81,yw82}. As seen in these works,
the algorithm constraining approach is successfully applied to problems of
image restoration, to smoothing in image reconstruction from projections (see
also \cite[Subsection 12.3]{h80}), and to constraining of linear iterative
processes in general. 
In \cite[Section 11.4]{h80} there is a discussion of all
kinds of, so called, \textquotedblleft tricks\textquotedblright\ that give
rise to what we call here \textquotedblleft constraining
operators.\textquotedblright\ This includes the trick of selective smoothing,
that is illustrated in detail in \cite[Section 5.3]{h80}. Historical
references can be found in \cite[p. 216]{h80}. The paper \cite{hl76} is the
original source of tricks in the field of image reconstruction from
projections.
A constraining strategy which applies a single strictly
nonexpansive idempotent operator at every iteration of the classical Kaczmarz
algorithm has been presented in \cite{kl90}. Recently \cite{npps}, {the third author} proposed
a generalization of this result by replacing Kaczmarz's algorithm by a more
general iterative process. Under the assumption that a family of strictly
nonexpansive operators $\{T_{k}\}_{k=0}^{\infty}$ has a nonempty common fixed
points set and an additional condition, reasonable in image reconstruction
problems, we proved that the sequence generated by the iterative scheme%
\begin{equation}
x^{0}\in R^{n}\text{ and }x^{k+1}=T_{k+1}(x^{k}),\text{ for all }k\geq0\text{,}%
\end{equation}
converges to a common fixed point of the operators $\{T_{k}\}_{k=0}^{\infty}$.

The particular problem of finding a common fixed point of nonlinear mappings
is an important topic in fixed point theory, see, e.g., the excellent recent
monograph \cite{cegielski-book}. We will denote by $F$ the set of common fixed
points. For finitely many paracontracting operators $T_{1},T_{2},\dots,T_{p}$
the following algorithm, proposed in \cite{ekn92},%
\begin{equation}
x^{0}\in R^{n}\text{ and }x^{k+1}=T_{j_{k+1}}(x^{k}),\text{ for all }k\geq0,\text{
with }\{j_{k}\}_{k=0}^{\infty}\text{ admissible},
\end{equation}
{converges if and only if $F$ is nonempty. Moreover, in this case the limit point of the sequence is an element of $F$}. The authors
introduced also a generalization of their result for a family $\{T_{k}%
\}_{k=0}^{\infty}$ consisting of finitely many subsequences convergent to
paracontracting operators. Our result in Theorem \ref{convergenta} is similar
in spirit to, but not identical with, \cite[Theorem 3]{ekn92}.

Under suitable assumptions, the convergence of the following algorithm,
proposed in \cite{b94},%
\begin{equation}
a,x^{0}\in R^{n}\text{ and }x^{k+1}=\alpha_{k+1}a+(1-\alpha_{k+1}%
)T_{k+1}(x^{k}),\text{ for all }k\geq0,\label{alg:b}%
\end{equation}
was investigated for a finite number of nonexpansive operators $T_{1}%
,T_{2},\dots,T_{p}$ activated cyclically, when $F\neq\emptyset$ and the
sequence $\{\alpha_{k}\}_{k=0}^{\infty}$ satisfies $\alpha_{k}\rightarrow0$,
$\sum_{k}{\left| \alpha_{k}-\alpha_{k+n}\right| }<+\infty$ and
$\sum_{k}{\alpha_{k}}=+\infty$. Bauschke \cite{b94} showed that the limit
point of any orbit generated by this algorithm is the projection $P_{F}(a)$ of
$a$ onto $F$. The question of finding $P_{F}(a)$ for a given $a$ is known as
the best approximation problem with respect to $F$.

Another approach in determining a common fixed point for a finite pool of
nonexpansive mappings was studied in \cite{kl97}. The authors examined the
convergence of an acceleration technique under various hypotheses. Their
method employs the construction of halfspaces at every iteration. The next
approximation is calculated by projecting the current one on a {surrogate halfspace}.

The following general algorithm was analyzed in \cite{c01} for a family
$\{T_{k}\}_{k=0}^{\infty}$ of firmly nonexpansive operators%
\begin{equation}
x^{0}\in R^{n}\text{ and }x^{k+1}=x^{k}+\alpha_{k+1}(T_{k+1}(x^{k}%
)-x^{k}),\text{ for all }k\geq0.
\end{equation}
Weak and corresponding strong convergence was established under various
assumptions on the sequence $\{\alpha_{k}\}_{k=0}^{\infty}$.

Hirstoaga \cite{h06} extended the results of \cite{b94} and showed under
suitable hypotheses the convergence of the algorithm%
\begin{equation}
x^{k+1}=\alpha_{k}Q(S_{k}x^{k})+(1-\alpha_{k})T_{k}x^{k},\text{ for all }k\geq0,
\end{equation}
where $\{T_{k}\}_{k=0}^{\infty}$ and $\{S_{k}\}_{k=0}^{\infty}$ are
quasi-nonexpansive, $Q$ is a strict contraction and $\{\alpha_{k}%
\}_{k=0}^{\infty}$ satisfies $\alpha_{k}\rightarrow0$ and $\sum_{k}{\alpha_k
}=+\infty$.

When solving the best approximation problem with respect to $F$ for a
uniformly asymptotically regular semigroup of nonexpansive operators,
\cite{ac05} introduced an algorithm similar to (\ref{alg:b}). Assuming that
$C$ is a convex subset of a real Hilbert space $H,$ that $G$ is an unbounded subset
of $R_{+}$, that $\{T_{t}\mid\text{ }t\in G\text{ and }T_{t}:H\rightarrow H\}$
is a uniformly asymptotically regular semigroup of nonexpansive operators with
$F\neq\emptyset$, that $\{\alpha_{k}\}_{k=0}^{\infty}$ is a steering sequence,
i.e., $\alpha_{k}\rightarrow0$, $\sum_{k}{\left| \alpha_{k}-\alpha
_{k+1}\right| }<+\infty$ and $\sum_{k}{\alpha_{k}}=+\infty$, and that
$\{r_{k}\}_{k=0}^{\infty}$ is an increasing unbounded sequence such that%
\begin{equation}
\sum_{k}{\operatorname{sup}_{x\in C}{\left\Vert T_{s}T_{r_{k+1}}%
(x)-T_{r_{k+1}}(x)\right\Vert }}<+\infty
\end{equation}
holds for all $s\in G$, it is proved in \cite{ac05} that the algorithm%
\begin{equation}
x^{0}\in H\text{ and }x^{k+1}=\alpha_{k+1}a+(1-\alpha_{k+1})T_{r_{k+1}}%
(x^{k}),\text{ for all }k\geq0,
\end{equation}
yields, for a given $a\in C,$ an approximation of $P_{F}(a)$, where $\parallel
\cdot\parallel$ is the induced norm.

\section{Convergence for a Family of Strictly Nonexpansive
Operators\label{sec0bis}}

We will prove in this section that, under two special hypotheses, an iterative
scheme which employs a family of {\textit{strictly nonexpansive (SNE)
operators}}$,$ i.e., operators that obey the next definition, converges to a
common fixed point.

{In the rest of the paper $\langle\cdot,\cdot\rangle$ and $\parallel
\cdot\parallel$ denote the Euclidean scalar product and norm, respectively, in
the $n$-dimensional Euclidean space ${R}^{n}$}.

\begin{definition}
\label{def:SNE}We say that an operator $T:R^{n}\rightarrow{R}^{n}$ is strictly
nonexpansive if, for all $x,y\in{R}^{n}$,%
\begin{equation}
\left\Vert T(x)-T(y)\right\Vert \leq\left\Vert x-y\right\Vert
,\label{def:SNE1}%
\end{equation}
and%
\begin{equation}
\text{if~}\left\Vert T(x)-T(y)\right\Vert =\left\Vert x-y\right\Vert
,~\text{then}~T(x)-T(y)=x-y.\label{def:SNE2}%
\end{equation}

\end{definition}

For a family $\{T_{k}\}_{k=0}^{\infty}$ of strictly nonexpansive operators we
define the fixed points sets and their intersection by%
\begin{equation}
\operatorname{Fix}(T_{k})=F_{k}:=\{x\in{R}^{n}\mid T_{k}(x)=x\}\text{ and
}F:=\cap_{k=0}^{\infty}F_{k},\label{Fk}%
\end{equation}
respectively, and assume that%
\begin{equation}
F\neq\emptyset.\label{F}%
\end{equation}
Consider the following algorithm.

\begin{algorithm}
\label{alg:gen}{\ }

\textbf{Initialization}\textit{:} $x^{0}\in{R}^{n}$ is arbitrary.

\textbf{Iterative step}\textit{: For every }$k\geq0,$ g\textit{iven the
current iterate }$x^{k}$ calculate the next iterate $x^{k+1}$ by%
\begin{equation}
x^{k+1}=T_{k+1}(x^{k}). \label{T_algoritm}%
\end{equation}

\end{algorithm}

\begin{remark}
Any sequence $\{x^{k}\}_{k=0}^{\infty},$ generated by Algorithm \ref{alg:gen},
is Fej\'{e}r monotone with respect to $F$ (see, e.g., \cite{c01}).
\end{remark}

The following two well-known results (see, e.g., \cite{bc11,c01}) will lead us
to the proof of convergence {of} Algorithm \ref{alg:gen}.

\begin{proposition}
\label{p6}Let $\{T_{k}\}_{k=0}^{\infty}$ be a family of strictly nonexpansive
operators for which (\ref{F}) holds and $z\in F$, then for any sequence
$\{x^{k}\}_{k=0}^{\infty},$ generated by Algorithm \ref{alg:gen}, the sequence
$\{\left\Vert x^{k}-z\right\Vert \}_{k=0}^{\infty}$ is decreasing.
\end{proposition}

\begin{corollary}
\label{cor2}Under the assumptions of Proposition \ref{p6}, any sequence
$\{x^{k}\}_{k=0}^{\infty},$ generated by Algorithm \ref{alg:gen}, is bounded.
\end{corollary}

We will make use of the following additional condition.

\begin{condition}
\label{ipoteza_suplimentara}Under the assumptions of Proposition \ref{p6}, if
$\{x^{k}\}_{k=0}^{\infty}$ is any sequence, generated by Algorithm
\ref{alg:gen}, then for every $\ell\geq0$, there exists an index $k(\ell
)\geq0$ such that%
\begin{equation}
\left\Vert T_{k+1}(x^{k})-z\right\Vert \leq\left\Vert T_{\ell}(x^{k}%
)-z\right\Vert ,
\end{equation}
for all $z\in F$ and all $k\geq k(\ell)$.
\end{condition}

\begin{remark}
\label{CP} Condition \ref{ipoteza_suplimentara} induces a kind of
\textquotedblleft monotonicity\textquotedblright\ concerning the sequence
$\{x^{k}\}_{k=0}^{\infty}$ generated by Algorithm \ref{alg:gen}. This becomes
clearer in Lemma \ref{vstarinf_nevida} in Section \ref{sec2} below, where the
assumption (\ref{conditie}) is invoked. It differs from the \textit{strong
attractivity with respect to }$F$ of a nonexpansive operator $T$ which is
defined in \cite[page 372]{bb96} by%
\begin{equation}
k\left\Vert Tx-x\right\Vert ^{2}\leq\left\Vert x-f\right\Vert ^{2}-\left\Vert
Tx-f\right\Vert ^{2},\label{CPE}%
\end{equation}
where $k$ is a positive constant.
\end{remark}

\begin{proposition}
\label{prop7}Under the assumptions of Proposition \ref{p6} and the assumption
that Condition \ref{ipoteza_suplimentara} holds, if $\overline{x}$ is an
accumulation point of a sequence $\{x^{k}\}_{k=0}^{\infty},$ generated by
Algorithm \ref{alg:gen}, then $\overline{x}\in F$.
\end{proposition}

\begin{proof}
The boundedness of $\{x^{k}\}_{k=0}^{\infty}$ that follows from Corollary
\ref{cor2}, guarantees the existence of $\overline{x}$. Let $\{x^{k_{s}%
}\}_{s=0}^{\infty}\subseteq\{x^{k}\}_{k=0}^{\infty}$ such that $\lim
_{s\rightarrow\infty}x^{k_{s}}=\overline{x}.$ Take any $\ell\geq0$ and let
$k(\ell)$ be as in Condition \ref{ipoteza_suplimentara}. There exists an
$s(\ell)\geq0$ such that $k_{s}\geq k(\ell)$, for all $s\geq s(\ell)$. 
{As $\{x^{k_{s}}\}_{s=0}^{\infty}$ is a subsequence of $\{x^{k}\}_{k=0}^{\infty}$, we have for 
every $s\geq0$}
\begin{equation}
k_{s+1}\geq k_{s}+1\geq k_{s}. \label{inegalitate_indici}%
\end{equation}
From (\ref{F}) there exists $z\in F$. Then, for $s\geq s(\ell)$, we have from
(\ref{inegalitate_indici}), Proposition \ref{p6}, (\ref{T_algoritm}),
Condition \ref{ipoteza_suplimentara} and (\ref{def:SNE1}),%
\begin{align}
\left\Vert x^{k_{s+1}}-z\right\Vert  &  \leq\left\Vert x^{k_{s}+1}%
-z\right\Vert =\left\Vert T_{k_{s}+1}(x^{k_{s}})-z\right\Vert \leq\left\Vert
T_{\ell}(x^{k_{s}})-z\right\Vert \nonumber\\
&  =\left\Vert T_{\ell}(x^{k_{s}})-T_{\ell}(z)\right\Vert \leq\left\Vert
x^{k_{s}}-z\right\Vert .
\end{align}
By taking limits in the last inequality, as $s\rightarrow\infty,$ we get%
\begin{equation}
\left\Vert \overline{x}-z\right\Vert \leq\left\Vert T_{\ell}(\overline
{x})-z\right\Vert \leq\left\Vert \overline{x}-z\right\Vert ,
\end{equation}
therefore $\left\Vert T_{\ell}(\overline{x})-z\right\Vert =\left\Vert
\overline{x}-z\right\Vert $ and, using (\ref{def:SNE2}), it follows that
$T_{\ell}(\overline{x})=\overline{x}$ implying $\overline{x}\in F_{l}$. Since
$\ell$ was arbitrarily chosen we obtain%
\begin{equation}
\overline{x}\in F, \label{limita_apartine}%
\end{equation}
which completes the proof.

\smartqed \end{proof}

We can now state our main convergence result, which follows directly from
Proposition \ref{prop7}.

\begin{theorem}
\label{convergenta}Under the assumptions of Proposition \ref{p6} and the
assumption that Condition \ref{ipoteza_suplimentara} holds, any sequence
$\{x^{k}\}_{k=0}^{\infty},$ generated by Algorithm \ref{alg:gen}, converges to
an element of $F$.
\end{theorem}

\begin{remark}\label{remark:paracontracting} 
{Replacing the strict nonexpansivity of the operators $\{T_{k}\}_{k=0}^{\infty}$ with the assumption that they belong to the wider class of paracontracting operators (see \cite[Definition 1]{ekn92}), the results stated in Proposition \ref{p6}, Proposition \ref{prop7} and, consequently, Theorem \ref{convergenta} still hold.
}
\end{remark}

We present in the next section the case when every $T_{k}$, with $k\geq0$, is
the composition of a constraining operator $S_{k}$ with an algorithmic
operator $Q$.

\section{The Family-Constrained Algorithm (FCA) \label{sec1}}

Many iterative algorithms are of, or can be cast into, the form of
\textit{one-step stationary iterations} (see, e.g., \cite[Chapter 10]{OR70}).
Given an algorithmic operator $Q:R^{n}\rightarrow{R}^{n}$, {the original iterative process may have the form}%
\begin{equation}
x^{k+1}=Q(x^{k}),\text{ for all }k\geq0,\label{eq:stati}%
\end{equation}
under various assumptions on $Q.$ Constraining such an algorithm with a family
of operators means that we desire to use instead of (\ref{eq:stati}) the
iterative process
\begin{equation}
x^{k+1}=S_{k}Q(x^{k}),\text{ for all }k\geq0,
\end{equation}
where $\{S_{k}\}_{k=0}^{\infty}$ is a family of operators $S_{k}%
:R^{n}\rightarrow{R}^{n},$ henceforth called \textit{constraining operators}.

If $Q:R^{n}\rightarrow{R}^{n}$ and $S_{k}:R^{n}\rightarrow{R}^{n}$, with
$k\geq0$, are strictly nonexpansive, we define the operators $T_{k}%
:R^{n}\rightarrow{R}^{n}$ by
\begin{equation}
T_{k}(x):=S_{k}Q(x),\text{ for all }k\geq0,\label{eq:Tk}%
\end{equation}
and prove that they are also strictly nonexpansive. The following result
extends \cite[Proposition 4]{DPI90}.

\begin{proposition}
\label{p4}For any $k\geq0,$ an operator $T_{k}$ as in (\ref{eq:Tk}), in which
$Q$ and $S_{k}$ {are strictly
nonexpansive}, has the following properties:%
\begin{equation}
\left\Vert T_{k}(x)-T_{k}(y)\right\Vert \leq\left\Vert x-y\right\Vert ,\text{
{for all }}x,y\in{R}^{n}, \label{prop:p4-1}%
\end{equation}
and%
\begin{equation}
\text{if }\left\Vert T_{k}(x)-T_{k}(y)\right\Vert =\left\Vert x-y\right\Vert
,\text{ then }T_{k}(x)-T_{k}(y)=Q(x)-Q(y)=x-y. \label{prop:p4-2}%
\end{equation}

\end{proposition}

\begin{proof}
To prove (\ref{prop:p4-1}) we use (\ref{def:SNE1}) to obtain%
\begin{equation}
\left\Vert T_{k}(x)-T_{k}(y)\right\Vert =\left\Vert S_{k}(Q(x))-S_{k}%
(Q(y))\right\Vert \leq\left\Vert Q(x)-Q(y)\right\Vert \leq\left\Vert
x-y\right\Vert . \label{prop:p4-3}%
\end{equation}
To prove (\ref{prop:p4-2}) {suppose that} we have equalities in (\ref{prop:p4-3}). Using
(\ref{def:SNE2}) we obtain%
\begin{equation}
T_{k}(x)-T_{k}(y)=Q(x)-Q(y)=x-y,
\end{equation}
which completes the proof. 
\smartqed \end{proof}

For $\{T_{k}\}_{k=0}^{\infty}$ defined according to (\ref{eq:Tk}), {with $\{S_{k}\}_{k=0}^{\infty}$ and $Q$ strictly nonexpansive}, Algorithm
\ref{alg:gen} may be written as a constrained algorithm.

\begin{algorithm}
\label{alg:conalg}\textbf{The Family-Constrained Algorithm (FCA)}

\textbf{Initialization}\textit{:} $x^{0}\in{R}^{n}$ is arbitrary.

\textbf{Iterative step}\textit{: For every }$k\geq0,$ g\textit{iven the
current iterate }$x^{k}$ calculate the next iterate $x^{k+1}$ by%
\begin{equation}
x^{k+1}=S_{k+1}Q(x^{k}). \label{algoritm}%
\end{equation}

\end{algorithm}

Proposition \ref{p4} and Theorem \ref{convergenta} yield that if assumptions
(\ref{F}) and Condition \ref{ipoteza_suplimentara} hold, then {any sequence generated by the Algorithm \ref{alg:conalg} converges to an element of $F$}.

We prove next that if $Q\in\mathcal{F}_{2}$ (see Definition \ref{F2} below)
and $S_{k}=S$, for all $k\geq0$, is a smoothing matrix, such as the one used
in \cite{DPI90}, then the family defined by (\ref{eq:Tk}) satisfies all our
hypotheses. We use the following definitions.

\begin{definition}
\label{F1:original}\cite[Definition 1]{DPI90} Let $\mathcal{F}_{1}$ be the set
of continuous operators $Q:R^{n}\rightarrow{R}^{n}$ that satisfy%
\begin{equation}
\left\Vert Q(x)-Q(y)\right\Vert \leq\left\Vert x-y\right\Vert ,{\ }\text{{for
all }}x,y\in{R}^{n}. \label{def:F1-1:original}%
\end{equation}
and%
\begin{align}
\text{If }\left\Vert Q(x)-Q(y)\right\Vert  &  =\left\Vert x-y\right\Vert
,\text{ then }\nonumber\\
Q(x)-Q(y)  &  =x-y\text{ and }\langle x-y,Q(y)-y\rangle=0.
\label{def:F1-2:original}%
\end{align}

\end{definition}

\begin{definition}
\label{F2}\cite[Definition 2]{DPI90} Let $\mathcal{F}_{2}$ be the set of
operators $Q\in\mathcal{F}_{1}$ with the property that for all $S\in
{R}^{n\times n}$ the function $g:{R}^{n}\rightarrow{R}$ defined by%
\begin{equation}
g(x):=\left\Vert x-SQ(x)\right\Vert ^{2} \label{eq:gi}%
\end{equation}
attains its unconstrained global minimum.
\end{definition}

It is clear that if $Q\in\mathcal{F}_{2}$, then $Q$ is strictly nonexpansive.
We show next that the family $\{S_{k}\}_{k=0}^{\infty}$ with $S_{k}=S$, for
all $k\geq0$, {where $S$ is a symmetric, stochastic, with positive diagonal matrix, is strictly nonexpansive. Such a matrix $S$} satisfies the two following properties%
\begin{equation}
\left\Vert Sx\right\Vert \leq\left\Vert x\right\Vert ,\text{ for all }x\in
{R}^{n},\label{cor1a}%
\end{equation}
and%
\begin{equation}
\left\Vert Sx\right\Vert =\left\Vert x\right\Vert \text{ implies that
}Sx=x,\label{cor1b}%
\end{equation}
(see \cite[Corollary 1]{DPI90}).
From (\ref{cor1a}) it follows that%
\begin{equation}
\left\Vert Sx-Sy\right\Vert =\left\Vert S(x-y)\right\Vert \leq\left\Vert
x-y\right\Vert ,\text{ for all }x,y\in{R}^{n}.
\end{equation}
Consequently, for $x,y\in{R}^{n}$, if%
\begin{equation}
\left\Vert Sx-Sy\right\Vert =\left\Vert x-y\right\Vert ,
\end{equation}
then, from (\ref{cor1b}),%
\begin{equation}
\left\Vert S(x-y)\right\Vert =\left\Vert x-y\right\Vert \text{ implies that
}S(x-y)=Sx-Sy=x-y.
\end{equation}

If $Q \in\mathcal{F}_{2}$ and $S$ is a symmetric, stochastic, with positive
diagonal matrix, then the set of fixed points of the operator $T:=SQ$ is not
empty (see \cite[Lemma 1]{DPI90}). {Finally, Condition
\ref{ipoteza_suplimentara} is trivial in the context of a single operator used
at every iteration of the Algorithm \ref{alg:conalg}}.

\section{Solving The Linear Least Squares Problem\label{sec2}}

We show in this section that a commonly used, in the field of image
reconstruction, algorithmic operator $Q$ obeys the conditions set forth in
Section \ref{sec1}. Consider the linear least squares (LLS) problem of seeking
a vector $x\in R^{n}$ such that%
\begin{equation}
\left\Vert Ax-b\right\Vert =\min\{\left\Vert Az-b\right\Vert \mid z\in
R^{n}\},\label{eq:LSS}%
\end{equation}
where the matrix $A$ is $m\times n$ and $b\in R^{m}$. We use the notations
$A^{T},$ $\mathcal{R}(A)$, $\mathcal{N}(A)$, for the transpose, range and null
space of $A$, respectively, and $LSS(A;b)$ and $x_{LS},$ for the set of all
least squares solutions and the minimal norm solution of (\ref{eq:LSS}), respectively.

We present in the sequel a general iterative method, introduced recently in
\cite{npps}, and prove that its algorithmic operator is strictly nonexpansive
and, moreover, belongs to $\mathcal{F}_{2}$. Let $T$ and $R$ be matrices of
dimensions $n\times n$ and $n\times m,$ respectively, having the following {three}
properties with respect to a given $m\times n$ matrix $A$:%
\begin{equation}
T+RA=I,\label{pr1}%
\end{equation}
where $I$ is the {identity} matrix;%
\begin{equation}
\text{for every }y\in R^{m}\text{ we have}~Ry\in{\mathcal{R}}(A^{T});\label{pr4}
\end{equation}%
\begin{equation}
\text{defining}~\tilde{T}:=TP_{{\mathcal{R}}(A^{T})}~\text{we have}%
~\left\Vert\tilde{T}\right\Vert<1,\label{pr5}%
\end{equation}
where {$P_V$ and $\Vert\tilde{T}\Vert$ denote the orthogonal projection onto a linear subspace $V$ and the induced norm of $\tilde{T}$, respectively}.

The following result is known, see, e.g., \cite{npps}.

\begin{proposition}
When $A$ and $b$ are as in (\ref{eq:LSS}) and the matrices $T$, $R$ and $A$
have the properties (\ref{pr1})--(\ref{pr5}), then the matrix $T$ has the
properties%
\begin{equation}
\text{if}~x\in{\mathcal{N}}(A)~\text{then}~Tx=x,\label{pr2}%
\end{equation}%
\begin{equation}
\text{if~}x\in{\mathcal{R}}(A^{T})~\text{then}~Tx\in{\mathcal{R}}%
(A^{T}),\label{pr3}%
\end{equation}%
\begin{equation}
\left\Vert Tx\right\Vert =\left\Vert x\right\Vert ~\text{if~~and~~only~~if}%
~x\in{\mathcal{N}}(A)\label{pr6}%
\end{equation}
and%
\begin{equation}
\left\Vert T\right\Vert ~\leq1.\label{pr7}%
\end{equation}

\end{proposition}

\begin{proposition}
\label{prop:Tx+Rb} When $A$ and $b$ are as in (\ref{eq:LSS}) and the matrices
$T$, $R$ and $A$ have the properties (\ref{pr1})--(\ref{pr5}), then the affine
operator $Q:{R}^{n}\rightarrow{R}^{n}$ defined by%
\begin{equation}
Q(\cdot):=T(\cdot)+Rb\label{operatorQ}%
\end{equation}
belongs to $\mathcal{F}_{2}$.
\end{proposition}

\begin{proof}
We first show that $Q\in\mathcal{F}_{1}$. To prove (\ref{def:F1-1:original})
let $x,y\in R^{n}$, then, from (\ref{pr7})%
\begin{equation}
\left\Vert (Tx+Rb)-(Ty+Rb)\right\Vert =\left\Vert T(x-y)\right\Vert
\leq\left\Vert T\right\Vert \left\Vert x-y\right\Vert \leq\left\Vert
x-y\right\Vert .
\end{equation}
Now for the case that $\left\Vert (Tx+Rb)-(Ty+Rb)\right\Vert =\left\Vert
x-y\right\Vert $ we obtain from (\ref{pr6}) that%
\begin{equation}
x-y\in{\mathcal{N}}(A). \label{apartine na}%
\end{equation}
Using (\ref{pr2}) we get%
\begin{equation}
T(x-y)=x-y\text{ if and only if }(Tx+Rb)-(Ty+Rb)=x-y.
\end{equation}
Representing $y$ as $y=P_{\mathcal{N}(A)}(y)+P_{\mathcal{R}(A^{T})}(y)$ and
using (\ref{pr2}) we obtain%
\begin{equation}
(Ty+Rb)-y=TP_{\mathcal{R}(A^{T})}(y)+Rb-P_{\mathcal{R}(A^{T})}(y),
\end{equation}
which, from (\ref{pr3}) and (\ref{pr4}), gives us%
\begin{equation}
(Ty+Rb)-y\in\mathcal{R}(A^{T}). \label{apartine rat}%
\end{equation}
Therefore, from (\ref{apartine na}) and (\ref{apartine rat}) it follows that%
\begin{equation}
\langle x-y,(Ty+Rb)-y\rangle=0.
\end{equation}

{
The second derivative of $g(\cdot)$, defined by (\ref{eq:gi}), is the constant function $g^{\prime\prime}(\cdot) = 2(I - SQ)^T(I - SQ)$. Since, for any $S\in{R}^{n\times n}$, the matrix $(I - SQ)^T(I - SQ)$ is symmetric and positive-semidefinite, it follows that $g$ is convex and attains its global minimum.
}
%Taking any matrix $S\in{R}^{n\times n},$ the function $g(\cdot)$ from
%(\ref{eq:gi}) is quadratic since $Q(\cdot)=T(\cdot)+Rb$ is an affine operator.
%The boundedness from below of $g(\cdot)$ from (\ref{eq:gi}) by zero,
%guarantees that it attains its minimum and the proof is complete. 
\smartqed \end{proof}

\begin{remark}
{
\label{rem14} The FCA Algorithm \ref{alg:conalg}, with $T_{k}$ as in
(\ref{eq:Tk}), $S_{k}=I$, $Q$ as in (\ref{operatorQ}) with $T,R$ as in
(\ref{pr1})-(\ref{pr5}) includes the Kaczmarz (see, e.g., \cite{tanabe}), Cimmino (see, e.g., \cite{npps}) and Diagonal Weighting (see, e.g., \cite{cpopaDW}) algorithms (for details and proofs of this statement see \cite{p12}). We will prove in the following result that another such example is the Landweber method (see, e.g., \cite{l51,p12}).
}
\end{remark}

\begin{proposition}
{
Let $\{\omega_{k}\}_{k=0}^{\infty} \subset R^n$ have the property that there exists a real $\epsilon$ such that 
\begin{equation}\label{omega}
0 < \epsilon \leq \omega_k \leq \frac{2}{\rho(A)^2} - \epsilon, 
\end{equation}
where $\rho(A)$ denotes the spectral norm of $A$. For any $x^0 \in R^n$ and $k \geq 0$ the Landweber iteration is defined by
\begin{equation}\label{landweber}
x^{k + 1} = (I - \omega_k A^T A)x^k + \omega_k A^T b.
\end{equation}
If we denote $I - \omega_k A^T A$ by $T_k$ and $\omega_k A^T$ by $R_k$, then, for every $k\geq 0$, the properties (\ref{pr1})-(\ref{pr5}) hold.
}
\end{proposition}
\begin{proof}
Let $k \geq 0$ be arbitrarily fixed. From the definitions of $T_k$ and $R_k$ it follows that (\ref{pr1}) and (\ref{pr4}) are satisfied. If we denote by $A^{\dagger}$ the (unique) Moore-Penrose pseudoinverse of $A$, then we have the relation $AA^{\dagger}A = A$ (see, e.g., \cite{b96}) and we may write $P_{{\mathcal{R}}(A^{T})} = A^{\dagger}A$ (see, e.g., \cite{bg03}). Consequently, according to (\ref{pr5}), we obtain
\begin{equation}\label{eq:T_landweber}
\tilde{T_k} = (I - \omega_k A^T A) A^{\dagger}A = A^{\dagger}A - \omega_k A^TAA^{\dagger}A = A^{\dagger}A - \omega_k A^T A.
\end{equation}
Consider the SVD decomposition $A = U \Sigma V^T$, where the matrices $U$, $\Sigma$ and $V$ are of dimensions $m \times m$, $m \times n$ and $n \times n$, respectively. We have that 
\begin{equation}\label{res:Sigma}
\Sigma = \left( \begin{array}{cc}
\Sigma_1 & 0 \\
0 & 0 
\end{array}\right), {\rm ~ with ~} \Sigma_1 = \operatorname{diag}\left(\sigma_1, \sigma_2, \dots, \sigma_r \right), 
\end{equation}
where $r$ is the rank of $A$ and $\sigma_1 \geq \sigma_2 \geq \dots \geq \sigma_r > 0$. Then, the pseudoinverse has the form (for details and proofs see \cite{b96})
\begin{equation}\label{res:pinv}
A^{\dagger} = V \left( \begin{array}{cc}
\Sigma_1^{-1} & 0 \\
0 & 0 
\end{array}\right) U^T.
\end{equation}
After a simple computation, (\ref{eq:T_landweber}), (\ref{res:Sigma}) and (\ref{res:pinv}) yield
\begin{equation}
\tilde{T_k}= V E V^T,
\end{equation}
where 
\begin{equation}
E = \left( \begin{array}{cc}
E_1 & 0 \\
0 & 0 
\end{array}\right), {\rm ~ with ~} E_1 = \operatorname{diag}\left(1 - \omega_k \sigma_1^2, 1 - \omega_k\sigma_2^2, \dots, 1 - \omega_k\sigma_r^2 \right).
\end{equation}
Therefore, since $\tilde{T_k}$ is normal, we obtain
\begin{equation}
\left\Vert \tilde{T_k} \right\Vert = \rho(V E V^T) = \rho(V^T V E) = \rho(E) = \max_{i \in \{1, 2, \dots, r\}} | 1 - \omega_k \sigma_i^2|,
\end{equation}
which together with (\ref{omega}) gives
\begin{equation}
\left\Vert \tilde{T_k} \right\Vert < 1,
\end{equation}
and the proof is complete.
\end{proof}

\begin{lemma}
\label{egalitate_multimi}Let $\operatorname{Fix}(Q)$ be the fixed points set
of the operator $Q$ defined by (\ref{operatorQ}), with $T$ and $R$ matrices of
dimensions $n\times n$ and $n\times m$, respectively, having the properties
(\ref{pr1})--(\ref{pr5}). The following property then holds%
\begin{equation}
\operatorname{Fix}(Q)=\{x+\Delta\mid x\in LSS(A;b)\}, \label{23}%
\end{equation}
where%
\begin{equation}
\Delta=(I-\tilde{T})^{-1}RP_{{\mathcal{N}}(A^{T})}(b). \label{delta}%
\end{equation}

\end{lemma}

\begin{proof}
For $x\in LSS(A;b)$ we know that%
\begin{equation}
Ax=P_{{\mathcal{R}}(A)}(b), \label{r1}%
\end{equation}
thus, by using (\ref{pr1}) we get%
\begin{equation}
(I-T)(x+\Delta)=RA(x+\Delta)=RP_{{\mathcal{R}}(A)}(b)+RA\Delta. \label{r2}%
\end{equation}
$\Vert\tilde{T}\Vert$ is the spectral norm of $\tilde{T}$, thus a matrix norm.
Since, by (\ref{pr5}), $\Vert\tilde{T}\Vert<1$, it follows from
\cite[Corollary 5.6.16 on page 301]{Horn} that $I-\tilde{T}$ is invertible and%
\begin{equation}
(I-\tilde{T})^{-1}=\sum_{k=0}^{\infty}{\tilde{T}^{k}}. \label{r3}%
\end{equation}
Using (\ref{r1}), (\ref{pr5}), (\ref{pr4}), (\ref{pr3}) and (\ref{r3}) we
obtain%
\begin{align}
RA\Delta &  =(I-T)\Delta=(I-\tilde{T}-TP_{{\mathcal{N}}(A)})(I-\tilde{T}%
)^{-1}RP_{{\mathcal{N}}(A^{T})}(b)\nonumber\\
&  =RP_{{\mathcal{N}}(A^{T})}(b)-TP_{{\mathcal{N}}(A)}\sum_{k=0}^{\infty
}{\tilde{T}^{k}}RP_{{\mathcal{N}}(A^{T})}(b)=RP_{{\mathcal{N}}(A^{T})}(b).
\end{align}
In view of (\ref{r2}) we then obtain%
\begin{equation}
x+\Delta=T(x+\Delta)+Rb,
\end{equation}
which implies that $\{x+\Delta\mid x\in LSS(A;b)\}\subseteq\operatorname{Fix}%
(Q)$.

For the reverse inclusion we consider $x\in\operatorname{Fix}(Q)$, i.e., $x=Tx+Rb$, which allows us to write%
\begin{equation}
P_{{\mathcal{R}}(A^{T})}(x)+P_{{\mathcal{N}}(A)}(x)=TP_{{\mathcal{R}}(A^{T}%
)}(x)+TP_{{\mathcal{N}}(A)}(x)+Rb.
\end{equation}
From (\ref{pr2}) and (\ref{pr4}) we get in the above equality%
\begin{align}
P_{{\mathcal{R}}(A^{T})}(x)  &  =\tilde{T}P_{{\mathcal{R}}(A^{T}%
)}(x)+Rb=\tilde{T}(\tilde{T}P_{{\mathcal{R}}(A^{T})}(x)+Rb)+ Rb\nonumber\\
&  =\cdots=\tilde{T}^{k}P_{{\mathcal{R}}(A^{T})}(x)+(\sum_{i=0}^{k-1}\tilde
{T}^{i})Rb.
\end{align}
By taking the limit as $k\rightarrow\infty,$ {from (\ref{pr5}) and (\ref{r3})} we arrive at%
\begin{equation}
P_{{\mathcal{R}}(A^{T})}(x)=(I-\tilde{T})^{-1}Rb=x_{LS}+\Delta,
\end{equation}
which, in turn, implies%
\begin{equation}
Ax=A(P_{{\mathcal{R}}(A^{T})}(x)+P_{{\mathcal{N}}(A)}(x))=Ax_{LS}%
+A\Delta=P_{{\mathcal{R}}(A)}(b)+A\Delta,
\end{equation}
thus, $x-\Delta\in LSS(A;b)$, i.e.,%
\begin{equation}
\operatorname{Fix}(Q)-\Delta\subseteq LSS(A;b),
\end{equation}
which is equivalent to
\begin{equation}
\operatorname{Fix}(Q)\subseteq LSS(A;b)+\Delta.
\end{equation}
Since we have also proved that $\Delta+LSS(A;b)\subseteq\operatorname{Fix}%
(Q)$, the equality
\begin{equation}
\operatorname{Fix}(Q)=LSS(A;b)+\Delta
\end{equation}
follows and the proof is complete. 
\smartqed \end{proof}

{For $Q$ chosen as in Proposition \ref{prop:Tx+Rb}, we present in the following an example of a family of strictly nonexpansive constraining
operators $\{S_{k}\}_{k=0}^{\infty}$ such that $\{T_{k}\}_{k=0}^{\infty}$, defined according to (\ref{eq:Tk}), satisfies (\ref{F}) and Condition \ref{ipoteza_suplimentara}, making it
applicable for the convergence theory of the FCA (Algorithm \ref{alg:conalg}).
}

The family consists of metric projection operators onto closed and convex sets
in ${R}^{n}$ with an additional \textquotedblleft image inclusion
assumption\textquotedblright. The metric projection operator $C$ onto the box
$[a,b]=[a_{1},b_{1}]\times\cdots\times\lbrack a_{n},b_{n}]\subset R^{n}$ is
defined by its $i$-th component, $i=1,2,\ldots,n,$ as%
\begin{equation}
(Cx)_{i}:=\left\{
\begin{array}
[c]{ll}%
x_{i}, & ~\text{~if~}~x_{i}\in\lbrack a_{i},b_{i}],\\
a_{i}, & ~\text{~if}~~x_{i}<a_{i},\\
b_{i}, & ~~\text{if}~~x_{i}>b_{i}.
\end{array}
\right.  \label{ex1}%
\end{equation}

It is known that such an operator is strictly nonexpansive (see, e.g.,
\cite{DPI90,p12}).

\begin{lemma}
\label{incluziune}Let $C$ and $\overline{C}$ be two metric projection
operators onto the boxes $[a,b]=[a_{1},b_{1}]\times\cdots\times\lbrack
a_{n},b_{n}]\subset R^{n}$ and $[\overline{a},\overline{b}]=[\overline{a}%
_{1},\overline{b}_{1}]\times\cdots\times\lbrack\overline{a}_{n},\overline
{b}_{n}]\subset R^{n},$ respectively, defined as in (\ref{ex1}). If the image
sets $[\overline{a},\overline{b}]\subset [a,b],$ then for any $y\in [\overline{a},\overline{b}]$ the
following inequality holds%
\begin{equation}
\left\Vert \overline{C}z-y\right\Vert \leq\left\Vert Cz-y\right\Vert
,~\text{{for all }}z\in R^{n}.
\end{equation}

\end{lemma}

\begin{proof}
From the inclusion $[\overline{a},\overline{b}]\subset [a,b]$ we get that 
\begin{equation}\label{lemma:incluziune1}
[\overline
{a}_{i},\overline{b}_{i}]\subset\lbrack a_{i},b_{i}] \text{ for all } i=1,2,\ldots ,n.
\end{equation} 
If $z\in R^{n}$ is arbitrarily fixed, from (\ref{lemma:incluziune1}) it results $P_{[\overline
{a}_{i},\overline{b}_{i}]}P_{[a_{i},b_{i}]}z_i = P_{[\overline
{a}_{i},\overline{b}_{i}]}z_i$, for all $i=1,2,\ldots ,n$. Therefore,
\begin{equation}\label{lemma:incluziune2}
\overline{C}Cz = P_{[\overline{a},\overline{b}]}P_{[a,b]}z = P_{[\overline{a},\overline{b}]}z = \overline{C}z.
\end{equation}
For any $y \in [\overline{a},\overline{b}]$, we have
\begin{equation}\label{lemma:incluziune3}
y = \overline{C}y.
\end{equation}
Since the linear mappings $C$ and $\overline{C}$ are strictly nonexpansive, (\ref{lemma:incluziune2}) and (\ref{lemma:incluziune3}) yield 
\begin{equation}
\left\Vert \overline{C}z-y\right\Vert = \left\Vert \overline{C}Cz-\overline{C}y\right\Vert \leq \left\Vert Cz-y\right\Vert
,~\text{{for all }}z\in R^{n}
\end{equation}
and the proof is complete.
\smartqed \end{proof}

Consider now a family $\{C_{k}\}_{k=0}^{\infty}$ of operators, where for each
$k\geq0$, $C_{k}$ is a metric projection operator onto the $k$-th box
$[a_{k},b_{k}]\subseteq R^{n},$ as defined in (\ref{ex1}). For this family we
define the sets%
\begin{equation}
{\mathcal{V}}_{k}^{\ast}:=\{z\in\text{Im}(C_{k})\mid~z-\Delta\in LSS(A;b)\},
\label{vstark}%
\end{equation}
and assume that for all $k\geq0,$%
\begin{equation}
{\mathcal{V}}_{k}^{\ast}\neq\emptyset. \label{vstark_nevide}%
\end{equation}

We develop next a sufficient condition for this family $\{C_{k}\}_{k=0}%
^{\infty}$ {to guarantee that the sequence $\{C_{k}Q\}_{k=0}^{\infty}$ satisfies (\ref{F}) and Condition
\ref{ipoteza_suplimentara}, where $Q$ is defined according to Proposition \ref{prop:Tx+Rb}}.

\begin{lemma}
\label{vstarinf_nevida}Let $\{C_{k}\}_{k=0}^{\infty}$ be a family of metric
projection operators onto the $k$-th box $[a_{k},b_{k}]\subset R^{n},$ as
defined in (\ref{ex1}) and assume that ${\mathcal{V}}_{k}^{\ast}\neq\emptyset$
for all $k\geq0$. If for every $\ell\geq0$ there exists a $k(\ell)\geq\ell$
such that%
\begin{equation}
\text{ }Im(C_{k+1})\subseteq Im(C_{\ell}),\text{{ for all }}k\geq k(\ell),
\label{conditie}%
\end{equation}
then the infinite intersection set%
\begin{equation}
{\mathcal{V}}_{\infty}^{\ast}:=\cap_{k=0}^{\infty}{\mathcal{V}}_{k}^{\ast},
\label{vstarinf}%
\end{equation}
is nonempty.
\end{lemma}

\begin{proof}
We construct a decreasing nested sequence of nonempty closed and bounded sets
in order to apply Cantor's Intersection Theorem (see, e.g., \cite{s63}).
Defining for every $\ell\in N$ the set $B_{\ell}:=\cap_{i=0}^{\ell
}{\mathcal{V}}_{i}^{\ast}$, it is clear that for every $\ell\geq0$,
$B_{\ell+1}\subseteq B_{\ell}$.
{Moreover, since $B_0 \subseteq [a_0, b_0]$, it is bounded}. 

Since $LSS(A;b)$ is closed and, for each $\ell\geq0$, $Im(C_{\ell})$
is closed, $B_{\ell}$ is also closed. Next we show that $B_{\ell}$ is nonempty
for every $\ell\geq0$. Take an arbitrarily fixed $\ell\in N$, and $k(i)\geq i$
for all $i\in\{0, 1, 2, \dots,\ell\},$ as in (\ref{conditie}), and define%
\begin{equation}
\overline{k}:=\max\{k(0), k(1),k(2), \dots, k(\ell)\}.
\end{equation}

Using the definition (\ref{vstark}) we obtain%
\begin{equation}
{\mathcal{V}}_{k+1}^{\ast}\subseteq{\mathcal{V}}_{i}^{\ast},\text{{ for all }
}k\geq\overline{k}~\text{and}~i\in\{0, 1, 2, \dots\ell\},
\end{equation}
which implies that%
\begin{equation}
\cap_{i=0}^{\ell}{\mathcal{V}}_{i}^{\ast}\neq\emptyset.
\label{intersectie_nevida}%
\end{equation}

Since $\cap_{k=0}^{\infty}{\mathcal{V}}_{k}^{\ast}=\cap_{\ell=0}^{\infty
}B_{\ell},$ applying Cantor's Intersection Theorem yields ${\mathcal{V}%
}_{\infty}^{\ast}\neq\emptyset$. 
\smartqed \end{proof}

\begin{remark}
\label{remark:question}The condition (\ref{conditie}) from the previous lemma
is equivalent to the following component-wise inequality on the sequences
$\{a_{k}\}_{k=0}^{\infty},\{b_{k}\}_{k=0}^{\infty}\subset R^{n}$: for all
$\ell\geq0$ there exists a $k(\ell)\geq\ell$ such that%
\begin{equation}
a_{\ell}\leq a_{k+1}\leq b_{k+1}\leq b_{\ell},\text{{ for all }}k\geq
k(\ell).
\end{equation}
\end{remark}

{
The metric projection operators, like those in (\ref{ex1}), are frequently used
for constraining purposes in image reconstruction problems that are formulated according to (\ref{eq:LSS}). 
As mentioned at the beginning of this paper, the idea of using iteration independent constraints was previously examined, see Subsection \ref{sec:prev}. Our purpose is to explore a procedure of adapting the constraining function at each step of the algorithm to obtain a better approximation of the scanned image. 
The meaning of (\ref{conditie}) in practice is that the image
of every constraining function should be built from a priori knowledge to
contain the exact solution (the original image), however, $\{\operatorname{Im}(C_{k})\}_{k=0}^{\infty}$ should not necessarily be a decreasing nested sequence.

In \cite{pp13} such a family of constraining operators is used to solve a Tomographic Particle Image Velocimetry (TomoPIV) problem (see \cite{ps09} for more details), which reduces to reconstructing a binary vector. The difficulty of this problem is to find the number and the   approximate location of the particles, corresponding to values of one in the solution. When applying a constant $[0,1]^{n}$ constraining interval, the obtained approximation usually contains ``ghost'' particles. The authors observed in the aforementioned paper that, if as the iterations progress, the intervals are focused on zero or one values by using an iteration-adaptive constraining process, these ``ghosts'' are eliminated and the correct number of particles is found.
}

\begin{proposition}
\label{prop:Ck} For a family $\{C_{k}\}_{k=0}^{\infty}$ of box constraining
operators like those in (\ref{ex1}) with the properties (\ref{vstark_nevide})
and (\ref{conditie}), and an operator $Q$ defined by (\ref{operatorQ}), with
$T$ and $R$ matrices having the properties (\ref{pr1})--(\ref{pr5}), the
assumption (\ref{F}) and Condition \ref{ipoteza_suplimentara} are satisfied.
\end{proposition}

\begin{proof}
For the first part we use Lemma \ref{vstarinf_nevida}. Let $z\in{\mathcal{V}%
}_{\infty}^{\ast}$. From the definition (\ref{vstarinf}) of ${\mathcal{V}%
}_{\infty}^{\ast}$ we get that for all $k\geq0$%
\begin{equation}
C_{k}(z)=z
\end{equation}
and that%
\begin{equation}
z-\Delta\in LSS(A;b)),
\end{equation}
which is equivalent to
\begin{equation}
z\in\Delta+LSS(A;b)),
\end{equation}
which, from (\ref{23}), implies that%
\begin{equation}
z\in\operatorname{Fix}(Q),
\end{equation}
thus $z\in F$. It follows that ${\mathcal{V}}_{\infty}^{\ast}\subseteq F$ and,
since ${\mathcal{V}}_{\infty}^{\ast}$ is nonempty, that also $F\neq\emptyset$.

To prove Condition \ref{ipoteza_suplimentara} we use Lemma \ref{incluziune}.
For an arbitrarily fixed $\ell\geq1$ and $k(\ell)$ from (\ref{conditie}) we
choose $y\in F$, $k\geq k(\ell)$ and $z=Q(x^{k})$, from (\ref{algoritm}).
Using (\ref{conditie}), the fact that $F\subseteq Im(C_{k+1})$ and Lemma
\ref{incluziune} we get%
\begin{equation}
\Vert C_{k+1}(Q(x^{k}))-y\Vert~~\leq~~\Vert C_{\ell}(Q(x^{k}))-y\Vert,
\end{equation}
which completes the proof. 
\smartqed \end{proof}

In conclusion, according to Proposition \ref{prop:Tx+Rb}, Lemma
\ref{egalitate_multimi}, Proposition \ref{prop:Ck} and Theorem
\ref{convergenta}, we may solve the linear least squares problem
(\ref{eq:LSS}) using Algorithm \ref{alg:conalg} with $Q$ defined by
(\ref{operatorQ}), when $T$ and $R$ matrices have the properties
(\ref{pr1})--(\ref{pr5}) and a family $\{C_{k}\}_{k=0}^{\infty}$ of box
constraining operators like those in (\ref{ex1}) satisfying the properties
(\ref{vstark_nevide}) and (\ref{conditie}).\medskip

\begin{acknowledgements} 
We thank Gabor Herman for valuable comments and the anonymous referees for their sugguestions which helped to improve the first version of the manuscript.
The work of Y.C. was partially supported by grant
number 2009012 from the United States-Israel Binational Science Foundation
(BSF) and by US Department of Army award number W81XWH-10-1-0170. 
\end{acknowledgements}

\end{document}